\documentclass[a4paper,12pt]{article}
\usepackage{amsmath, amsthm}
\usepackage{mathtools}
\usepackage{amssymb}
\usepackage{amscd}
\usepackage{epic, eepic}
\usepackage{url}
\usepackage{color}
\usepackage[T2A]{fontenc}   
\usepackage[utf8]{inputenc} 
\usepackage{comment}
\usepackage{todonotes}
\usepackage{graphicx}
\usepackage{epstopdf}
\usepackage{enumerate}
\usepackage{array}
\usepackage{tabularx}
\usepackage{tikz}
\usepackage{enumitem}
\usepackage{tcolorbox}
\usepackage{pdflscape}
\usepackage{svg}
\usepackage{hyperref}

\usepackage{caption}
\usepackage{subcaption}

\newtheorem{theorem}{\bf Theorem}[section]
\newtheorem{lemma}[theorem]{\bf Lemma}
\newtheorem{cor}[theorem]{\bf Corollary}

\newtheorem{nota}[theorem]{\bf Notation}

\newtheorem{defi}[theorem]{\bf Definition}
\newtheorem{obs}[theorem]{\bf Observation}
\newtheorem{conjecture}[theorem]{\bf Conjecture}

\title{Extending partial edge-colorings of bounded size in Cartesian products of graphs}

\date{}

\author{Pál Bärnkopf \thanks{Alfréd Rényi Institute of Mathematics, Budapest, Hungary. Partially supported by the Counting in Sparse Graphs Lendület Research Group.
E-mail: {\tt barpal@renyi.hu}}
\and 
Ervin Győri \thanks{Alfréd Rényi Institute of Mathematics, Budapest, Hungary, Partially supported by the National Research, Development and Innovation Office NKFIH, grants K132696 and SNN135643, 	E-mail: {\tt gyori@renyi.hu}}
}

\begin{document}

\maketitle

\begin{abstract}

This paper studies edge-precoloring extensions in Cartesian products of graphs, motivated by a conjecture of Casselgren, Petros, and Fufa. We formulate a general hypothesis stating that if every edge-precoloring of $G$ and $H$ of sizes $k<\chi'(G)$ and $l<\chi'(H)$, respectively, is extendable, then any edge-precoloring of $G \square H$ of size $k+l+1$ can be extended to a proper $(\chi'(G)+\chi'(H))$-coloring. We provide partial progress toward this conjecture by establishing the result in cases where $k<\Delta(G)$, $G$ is a triangle-free $r$-regular graph and $H$ is a star, an even cycle, a path or, more generally, an arbitrary tree $F$. Furthermore, we prove the conjecture in the case where $G$ is a subcubic graph and $H = K_2$.

\textit{Keywords: Precoloring extension; edge-coloring; Bipartite graph; Cartesian product, Hypercube} 
\end{abstract}

\section{Introduction}

In this paper, we deal with proper edge-colorings of graphs. Throughout the paper, we often say just coloring instead of proper edge-coloring. The basic concepts not defined in the article can be found in the book \cite{diestel}. The edge chromatic number (or chromatic index) of a graph $G$ is denoted by $\chi'(G)$. According to Vizing's Theorem \cite{Vizing}, if $G$ is a simple graph with maximum degree $\Delta(G)$, then its chromatic index is either $\Delta(G)$ or $\Delta(G)+1$. For bipartite graphs, it is known that $\chi'(G) = \Delta(G)$ (König's edge-coloring theorem \cite{konig}). An (edge) \textit{precoloring} (or partial edge-coloring) of a graph $G$ is a proper edge-coloring of some edge set $E_0 \subseteq E(G)$. 

\begin{defi}
    \textup{We call a precoloring of some edge set $E_0$ \textit{extendable} if there is a proper edge-coloring with $\chi'(G)$ colors such that the color of the edges in $E_0$ are the prescribed colors. Such a coloring is called an \textit{extension} of the precoloring.}
\end{defi}


\begin{defi}
    \textup{The \textit{Cartesian product $G_1 \square G_2$ of graphs} $G_1=(V_1,E_1)$ and $G_2=(V_2,E_2)$ is the graph whose vertex set is the Cartesian product $V(G_1) \times V(G_2)$ and two vertices $(u_1,u_2)$ and $(v_1,v_2)$ are adjacent in $G_1 \square G_2$ if and only if either}
    \begin{itemize}
        \item \textup{$u_1=v_1$ and $u_2$ is adjacent to $v_2$ in $G_2$ or}
        \item \textup{$u_2=v_2$ and $u_1$ is adjacent to $v_1$ in $G_1$.}
    \end{itemize}
\end{defi}

\begin{nota}
    \textup{$G^d$ denotes the $d$-th power of the Cartesian product of $G$ with itself ($G^d = \underbrace{G \square G \square \cdots \square G}_{d}$).}
\end{nota}

\begin{nota}
    \textup{The $d$-dimensional hypercube, denoted $Q_d$, is the $d$-th power of the Cartesian product of $K_2$ with itself, i.e. $Q_d=K_2^d$.}
\end{nota}

Completion of partial (edge-)colorings of graphs has a long history. For instance, completion of partial Latin squares can be interpreted as an edge-coloring extension problem restricted to complete bipartite graphs, and this has been studied since as early as 1960; see, e.g. \cite{Smetaniuk}. The first known publication explicitly deals with edge-coloring extensions is by Marcotte and Seymour~\cite{Marcotte}, who studied when a particular necessary condition for extendability of a partial edge-coloring is also sufficient. Since then, it has been shown that the problem of extending a given edge-precoloring is an NP-hard problem, even for 3-regular bipartite graphs \cite{Easton, Fiala}.


Motivated by the result on hypercubes \cite{hypercube}, Casselgren, Petros and Fufa \cite{cartesian} extended the study of edge-precoloring extension of graphs with a particular focus on Evans-type questions for Cartesian products of graphs. Their primary interest is the following conjecture, which would be a far-reaching generalization of the main result of \cite{hypercube}.

\begin{conjecture} \label{basic} \cite{cartesian}
    If $G$ is a graph that every precoloring of at most $k < \chi'(G)$ edges can be extended to a proper $\chi'(G)$-edge-coloring, then every precoloring of at most $k+1$ edges of $G \square K_2$ is extendable to a proper $(\chi'(G)+1)$-edge-coloring of $G \square K_2$.
\end{conjecture}

They verify that this conjecture holds for trees, complete and complete bipartite graphs, as well as for graphs with small maximum degree. As a first step, we verify this conjecture in the case $k < \Delta(G)$ for triangle-free, regular graphs. This result implies a similar result replacing $K_2=K_{1,1}$ with $K_{1,n}$. Furthermore, we prove a similar statement for an even cycle (and consequently for any path). Based on the results, we formulate the following generalization of Conjecture \ref{basic}.

\begin{conjecture} \label{general}
    If $G$ is a graph that every precoloring of at most $k < \chi'(G)$ edges can be extended to a proper $\chi'(G)$-edge-coloring and $H$ is a graph that every precoloring of at most $l < \chi'(H)$ edges can be extended to a proper $\chi'(H)$-edge-coloring, then every precoloring of at most $k+l+1$ edges of $G \square H$ is extendable to a proper $(\chi'(G)+\chi'(H))$-edge-coloring of $G \square H$.
\end{conjecture}

Furthermore, Casselgren, Petros, and Fufa~\cite{cartesian} proved Conjecture~\ref{basic} for Class 1 graphs of maximum degree three. However, they noted that their method does not extend to Class 2 graphs. We prove Conjecture~\ref{basic} for Class 2 graphs of maximum degree three.

\section{Triangle-free $r$-regular graphs}

Before turning to the statements and their proofs, it is useful to introduce some notation and mention a few observations that we will use repeatedly in several of the proofs.

\begin{obs}
    \textup{Since we precolor at most $\Delta(G)$ edges in $G \square K_2$, and we color with $\chi'(G)+1$ colors, there necessarily exists at least one color that is not used in the precoloring. Let us denote the available colors as $\{1, \dots, \chi'(G)+1\}$ and (unless otherwise stated) let the color that is not used in the precoloring be $\chi'(G)+1$.}
\end{obs}

\begin{nota}
    \textup{For each vertex $v \in V(H)$, let $G_v$ be the subgraph of $G \square H$ induced by the vertices $\{(u,v) \mid u \in G\}$ (this is isomorphic to $G$) and for each edge $e=uv \in E(H)$, let $M_e$ be the set of edges $\{ (x,u)(x,v) \mid x \in G\}$.}
\end{nota}

\begin{theorem} \label{edge}
    If $G$ is a triangle-free, $r$-regular graph where every precoloring of at most $k < r$ edges can be extended to a proper $\chi'(G)$-edge-coloring, then every precoloring of at most $k+1$ edges of $G \square K_2$ is extendable to a proper $(\chi'(G)+1)$-edge-coloring of $G \square K_2$.
\end{theorem}

\begin{proof}
    Let $V(K_2)=\{a,b\}$. If every precolored edge is neither in $G_a$ or in $G_b$ and there are precolored edges in both $G_a$ and $G_b$ (and the color $\chi'(G)+1$ is not prescribed anywhere), we color $G_a$ and $G_b$ respecting the prescriptions using colors $\{1, \dots, \chi'(G)\}$. (This is possible by the hypothesis.) Then, color the edges running between $G_a$ and $G_b$ with the color $\chi'(G)+1$.

    If all precolored edges are in $G_a$ (or symmetrically, all are in $G_b$), we ignore the precoloring on an arbitrary precolored edge $e$ and color $G_a$ with colors $\{1, \dots, \chi'(G)\}$ such that the prescriptions on the remaining precolored edges are satisfied. (This is possible by the hypothesis since there are now at most $k$ constraints). After this, we recolor $e$ with its originally prescribed color. If there are edges adjacent to $e$ that currently have this color, we change their color to $\chi'(G)+1$. With this step, we only modify the colors of non-precolored edges, because no two adjacent edges can be precolored with the same color. We then copy the coloring of $G_a$ to $G_b$. Since $r$ edges of different colors are incident to every vertex $x$, there is still a free color available to color the edges of type $(x,a)(x,b)$.

    If an edge of type $(x,a)(x,b)$ is precolored, we choose a vertex $y$ such that $xy \in E(G)$ and neither $(y,a)$ nor $(y,b)$ is incident to any precolored edge. Such a vertex exists because $x$ has $r$ neighbors, while at most $r-1$ other edges are precolored, and (due to triangle-freeness) every precolored edge has at most one endpoint adjacent to $x$. We then precolor the edges $(x,a)(y,a)$ and $(x,b)(y,b)$ with the color originally assigned to $(x,a)(x,b)$. We can do this sequentially for every precolored edge of type $(x,a)(x,b)$. (Technically, the number of precolored edges increases, but since these induce identical prescriptions in the two copies of $G$, this does not cause any problem.) Finally, prescriptions remain only within $G_a$ and $G_b$. After performing the above steps, two cases may occur again.

    If there are at most $k$ prescriptions in both copies of $G$, we color both copies respecting the precoloring using colors $\{1, \dots, \chi'(G)\}$, and color the edges running between the two copies of $G$ with the color $\chi'(G)+1$. If an edge $(x,a)(x,b)$ had a prescription, the edges of the square $(x,a)(x,b)(y,a)(y,b)$ are now colored with two colors. We can swap the colors on the edges of this square (where $y$ is the vertex chosen for $x$ in the previous step). Since these squares are disjoint (due to the careful choice of the $y$-vertices), these swaps do not spoil the coloring, resulting in a proper coloring that is consistent with the original precoloring.

    Finally, if in one copy of $G$ (say $G_a$) the colors of $k+1$ edges are prescribed after projecting the precolored cross-edges, then $G_b$ only contains precoloring that arose during the projection (meaning it is also prescribed in $G_a$). In this case, let the color $\chi'(G)+1$ be a color that is prescribed on some edge of type $(x,a)(x,b)$. Do not project the edges precolored with $\chi'(G)+1$ onto $G_a$, and if there is an edge in $G_a$ precolored with $\chi'(G)+1$, ignore that prescription. Thus, at most $k$ edges will be precolored in $G_a$, meaning we can color it properly. Copy this coloring to $G_b$, then perform the color swaps on the squares as in the previous paragraph (if necessary). Furthermore, if an edge $(x,a)(y,a) \in G_a$ was precolored with $\chi'(G)+1$, we also swap the colors on the edges of the square $(x,a)(x,b)(y,a)(y,b)$. (These swaps still occur on disjoint squares.) Thus, in every case, we obtain a coloring of $G \square K_2$ that satisfies the precoloring conditions.
\end{proof}

\begin{cor} \label{star}
    If $G$ is a triangle-free, $r$-regular graph where every precoloring of at most $k < r$ edges can be extended to a proper $\chi'(G)$-edge-coloring, then every precoloring of at most $k+n$ edges of $G \square K_{1,n}$ is extendable to a proper $(\chi'(G)+n)$-edge-coloring of $G \square K_{1,n}$.
\end{cor}

\begin{proof}
    If $G$ is triangle-free and regular, then so is $G \square K_2$. Consequently, by repeatedly applying Theorem \ref{edge}, we obtain that any precoloring of at most $k+n$ edges of $G \square Q_n$ is extendable to a proper $(\chi'(G)+n)$-edge-coloring, where $Q_n$ denotes the $n$-dimensional hypercube. Since $G \square K_{1,n}$ is a subgraph of $G \square Q_n$, the assertion holds for this graph as well.
\end{proof}


\begin{lemma} \label{extra-color}
    Suppose $G$ is a graph such that any precoloring of a set of $k$ edges, is extendable to a proper $l$-edge-coloring of $G$. Then, for any $m \ge 0$, any precoloring of $k+m$ edges of $G$ is extendable to a proper $(l+m)$-edge-coloring of $G$.
\end{lemma}

\begin{proof}
    We prove the lemma by induction on $m$. The statement for $m=0$ is equivalent to the hypothesis (any precoloring of $k$ edges is extendable to a proper $l$-edge-coloring), so the statement holds.

    Assume that the statement holds for $m$. We now prove that the statement holds for $m+1$. Let $\varphi$ be a precoloring on a subset $L$ of the edges of $G$ such that $|L|=k+(m+1)$. The prescribed colors are taken from the set $C = \{1, \dots, l+m+1\}$.

    Choose an arbitrary edge $e \in L$, and let $c^* = \varphi(e)$ be its prescribed color. Let $L' = L \setminus \{e\}$ and let $\varphi'$ be the restriction of the precoloring $\varphi$ to $L'$. Since $|L'| = k+m$, there must be at least one color $c \in C$ that is not used by $\varphi'$. Let $C' = C \setminus \{c\}$. (The set $C'$ has $|C'| = \chi'(G)+m$ colors.)

    By the induction hypothesis, the precoloring $\varphi'$ on $L'$ is extendable to a proper edge-coloring using the colors in $C'$. If the resulting coloring colors $e$ with $c^*$, then the extension is successful and we are done. If $e$ is not colored $c^*$, then change the color of $e$ to $c^*$ and if any neighbor $e'$ of $e$ is colored $c^*$, change the color of $e'$ to $c$.

    This operation maintains the property of a proper coloring. The edge $e$ is properly colored since no neighbor of $e$ remains colored $c^*$. The edges that were newly colored $c$ cannot be adjacent to each other, because they were all (both) colored $c^*$ in the previous proper coloring (and therefore cannot be adjacent). The resulting coloring respects the precoloring $\varphi$ on $L$, because the color of $e$ has been corrected to $c^*$. Furthermore, the other edges whose colors were modified (from $c^*$ to $c$) could not have been precolored, since $\varphi(e)=c^*$, and a proper precoloring cannot assign $c^*$ to any neighbor of $e$. 

    Thus, we have successfully extended the precoloring $\varphi$ on $k+(m+1)$ edges using $l+m+1$ colors. By induction, the statement holds for all $m \ge 0$.
\end{proof}

\begin{theorem} \label{cycle}
    If $G$ is a triangle-free, $r$-regular graph such that any precoloring of at most $k < r$ edges can be extended to a proper $\chi'(G)$-edge-coloring, then any precoloring of at most $k+2$ edges of the graph $G \square C_{2m}$ $(m>1)$ is extendable to a proper $(\chi'(G)+2)$-edge-coloring of $G \square C_{2m}$.
\end{theorem}

\begin{proof}
    Let the vertices of $C_{2m}$ be denoted by $v_1, v_2, \dots, v_{2m}$, and its edges by $e_i = v_i v_{i+1}$ (where $v_{2m+1} = v_1$). Let $L$ be the set of precolored edges such that $|L| \le k+2$, and let $\varphi$ be the given precoloring.

    Consider two disjoint perfect matchings of the cycle $C_{2m}$: $P_1 = \{e_1, e_3, \dots, e_{2m-1}\}$ and $P_2 = \{e_2, e_4, \dots, e_{2m}\}$. The corresponding edge sets are $M_{P_1} = \bigcup_{e \in P_1} M_e$ and $M_{P_2} = \bigcup_{e \in P_2} M_e$ in $G \square C_{2m}$. It is clear that after removing either $M_{P_1}$ or $M_{P_2}$, the components of the remaining graph are $G \square P_2$ (or $G \square P_1$), which are unions of disjoint $G \square K_2$ components.

    If there exists a vertex $v_i \in C_{2m}$ such that $|L \cap G_{v_i}|=k+2$, we color $G_{v_i}$ according to the precoloring using $\chi'(G)+2$ colors (this is possible by Lemma \ref{extra-color}). Subsequently, for all $i \neq j$, we color $G_{v_j}$ identical to this coloring. (The subgraphs $G_{v_j}$ are all isomorphic to $G$.) Thus, for every vertex $u \in G$, the same $r$ colors are used, leaving two available colors. The edges of form $(u,v_j)(u,v_{j+1})$ constitute an even cycle, which can be colored using these two remaining colors.

    If there is no vertex $v_i \in C_{2m}$ such that $|L \cap G_{v_i}|=k+2$, we remove one of the sets $M_{P_1}$ or $M_{P_2}$ (denoted $M_{P_i}$) such that for every edge $e=uv \in P_i$, the condition $|L \cap (E(G_u) \cup E(M_e) \cup E(G_v))| \leq k+1$ holds. Such a choice is always possible; if both edges $e \in P_1$ and $f \in P_2$ violated this condition, then (since $M_e \cap M_f = \emptyset$) they would share a common endpoint $v$, and all $k+2$ precolorings would have to lie within $G_v$, which we have already settled.

    For every edge $f_x = ((x,u), (x,v)) \in L \cap M_{P_i}$ (with color $c$), we can find an edge $xy_x \in E(G)$ such that $g_u = ((z,u), (y_x,u)) \notin L$ and $g_v = ((z,v), (y_x,v)) \notin L$ for all $z \in V(G)$. This is guaranteed because for every edge $e=uv$, initially $|L \cap (E(G_u) \cup E(M_e) \cup E(G_v))| \leq k+1 \le r$. Furthermore, since $P_i$ is a matching, when a color prescription is projected into a subgraph $G_x$, its origin from a specific $M_e$ is unique. Although $|L \cap (E(G_u) \cup E(M_e) \cup E(G_v))|$ may increase, the number of constraints does not change effectively if we treat the edges $(x,u)(y,u)$ and $(x,v)(y,v)$ as a single unit (one prescription is replaced by two equivalent ones). Using the facts that $G$ is triangle-free (implying any precolored edge is incident to at most one neighbor of $x$) and $r$-regular, there must exist a neighbor $y_x$ of $x$ such that no precolored edge is incident to it in either $G_u$ or $G_v$. We then precolor the edges $g_u$ and $g_v$ associated with $x$ with color $c$. This process can be repeated iteratively to project all prescriptions such that no vertex is incident to more than one prescription originating from the projection.

    As a result, every component is of the form $G \square K_2$. If every component contains at most $k+1$ precolored edge after the projections, we can color them using $\chi'(G)+1$ colors such that the same colors are used in every component, per Theorem \ref{edge}, since there is a color that is not used in any precoloring. The edges of $M_{P_i}$ are then colored with this unused color. This yields a proper coloring of the graph that satisfies all original precolorings except for those in $M_{P_i}$, while also satisfying the projected prescriptions. The squares defined by the projected prescriptions, $((x,u),(y_x,u),(x,v),(y_x,v))$, are disjoint because $P_i$ (and thus $M_{P_i}$) is a matching and each $y_x$ was chosen to be free of other projected prescriptions. Each such square is colored with two colors. By swapping the colors along these squares, the coloring remains proper, does not violate the precoloring of any other edges, and ensures the color of the precolored edges in $M_{P_i}$ receive their designated colors, providing an extension of $\varphi$.

    Finally, we examine the case where the removal of $M_{P_i}$ leads to a component having $k+2$ requirements (after projecting the removed precolorings). This can only occur if for some $i$, $L \cap (M_{v_{i-1}v_i} \cup G_{v_i} \cup M_{v_iv_{i+1}})=k+2$. (In every other case, we could take the other one, from $M_{P_1}$ and $M_{P_2}$.) In this case, consider the same color prescriptions in the graph $G \square C_4$ by mapping $v_i$ to $w_2$ in $C_4$. ($V(C_4)=\{w_1,w_2,w_3,w_4\}$) This precoloring can be extended by applying Theorem \ref{edge} twice. We then copy this coloring into the graph $G \square C_{2m}$ and complete it according to Table \ref{tab:color_mapping}, which yields a proper extension of $\varphi$ in $G \square C_{2m}$. (The indices are taken modulo $2m$.)

    \begin{table}[h]
        \centering
        \begin{tabular}{|l|l|}
        \hline
        \textbf{Coloring of $G \square C_{2m}$} & \textbf{Based on which part of $G \square C_4$} \\ \hline
        $G_{v_{i-1}}$        & $G_{w_1}$           \\ \hline
        $G_{v_i}$        & $G_{w_2}$           \\ \hline
        $G_{v_{i+1}}$        & $G_{w_3}$           \\ \hline
        $G_{v_j}, j \notin \{i-1,i,i+1\}$        & $G_{w_4}$           \\ \hline
        $M_{v_{i-1}v_i}$        & $M_{w_1w_2}$           \\ \hline
        $M_{v_iv_{i+1}}$        & $M_{w_2w_3}$           \\ \hline
        $M_{v_{i+(2k+1)}v_{i+(2k+2)}}, k=0, \dots m-2$        & $M_{w_3w_4}$           \\ \hline
        $M_{v_{i+(2k+2)}v_{i+(2k+3)}}, k=0, \dots m-2$        & $M_{w_4w_1}$           \\ \hline
        \end{tabular}
        \caption{Mapping of the coloring}
        \label{tab:color_mapping}
    \end{table}
    
\end{proof}

\begin{cor} \label{path}
    If $G$ is a triangle-free, $r$-regular graph where every precoloring of at most $k < r$ edges can be extended to a proper $\chi'(G)$-edge-coloring, then every precoloring of at most $k+2$ edges of $G \square P_{m}$ $(m>1)$ is extendable to a proper $(\chi'(G)+2)$-edge-coloring of $G \square P_{m}$.
\end{cor}

\begin{proof}
    The graph $P_m$ is a subgraph of $C_{2m}$. According to Theorem \ref{cycle}, any precoloring of at most $k+2$ edges of $G \square C_{2m}$ is extendable to a proper $(\chi'(G)+2)$-edge-coloring. Since $G \square P_m$ is a subgraph of $G \square C_{2m}$, the assertion holds for this graph as well.
\end{proof}

\begin{theorem}
    Let $G$ be an $r$-regular, triangle-free graph and $F$ be a tree. If any precoloring of at most $k < r$ edges of $G$ can be extended to a proper $\chi'(G)$-edge-coloring of $G$, then any precoloring of at most $k+\Delta(F)$ edges of $G \square F$ can be extended to a proper $(\chi'(G)+\Delta(F))$-edge-coloring of $G \square F$.
\end{theorem}

\begin{proof}
    We proceed by induction on $|E(F)|$. If $|E(F)| \leq 2$, the assertion holds by Theorem \ref{edge} and Corollary \ref{path}.

    If $F$ is a star graph, the statement holds, due to Corollary \ref{star}; thus, we henceforth assume that $F$ is not a star.

    Assume that the statement holds for every tree $F'$ such that $|E(F')| < m$. Let $F$ be a tree with $|E(F)|=m$, and let $\varphi$ be a precoloring of a set $L$ of $k+\Delta(F)$ edges in the graph $G \square F$. 

    Suppose that there exists a leaf $v \in V(F)$ (with neighbor $u$ and incident edge $e=uv$) such that $L \cap (E(G_v) \cup E(M_e)) = \emptyset$. Let $F' = F-v$. If $\Delta(F') = \Delta(F)$, the assertion follows from the induction hypothesis; if $\Delta(F') = \Delta(F)-1$, it follows from the induction hypothesis combined with Lemma \ref{extra-color} that the precoloring $\varphi$ can be extended to $G \square F'$ using $\chi'(G)+\Delta(F)$ colors. Let us denote this coloring $\overline{\varphi}$.

    We define the coloring of $G_v$ as follows: for every edge $((x,v),(y,v)) \in E(G_v)$, let the color of $(x,v)(y,v)$ be $\overline{\varphi}((x,u)(y,u))$, thereby duplicating the coloring of $G_u$. This yields a proper edge-coloring of $G_v$ and does not conflict with $\varphi$, as no precoloring was specified on $G_v$.

    The coloring of $M_e$ is defined as follows: for an edge $(x,u)(x,v) \in M_e$, at most $\chi'(G) + \Delta(F)-1$ colors are used at vertex $(x,u)$, and no additional colors are used at vertex $(x,v)$ due to the duplicated coloring. As the total palette consists of $\chi'(G)+\Delta(F)$ colors, there is at least one available color for each such edge. Since $\varphi$ is empty on $M_e$ and $M_e$ is a matching, a proper coloring of $M_e$ is obtained by selecting an available color for each edge.

    Now consider the case where for every leaf $v$ of $F$, $L \cap (E(G_v) \cup E(M_e)) \neq \emptyset$. Let $P \subseteq E(F)$ be a matching that covers all vertices of maximum degree in $F$, such that at least one of its edges is not incident to any leaf. Such a matching $P$ exists: it is well known that every bipartite graph possesses a matching covering all vertices of maximum degree. If a matching consisted solely of edges incident to leaves, we could replace an edge incident to a maximum degree vertex with another incident edge not connected to a leaf (which exists since $F$ is not a star). If this remains a matching, we are done; otherwise, by removing the edge adjacent to the newly selected one, we obtain an appropriate matching. (The removed edge’s other endpoint was a leaf, so its removal does not affect the coverage of maximum degree vertices.) Let $M_F = \bigcup_{e \in P} M_e$ be the union of matchings in $G \square F$ corresponding to $P$. Then $F' = F \setminus P$ is a forest with components $F'_j$, where $\Delta(F'_j) \le \Delta(F)-1$.

    In the next step, we project the color prescriptions of $M_F$ onto the trees $F'_j$ as follows:

    For every edge $f_x = (x,u)(x,v) \in L \cap M_F$ (with color $c$), we find an edge $xy_x \in E(G)$ such that $g_u = (z,u)(y_x,u) \notin L$ and $g_v = (z,v)(y_x,v) \notin L$ for all $z \in V(G)$. This is possible because every leaf $v$ satisfies $L \cap (E(G_v) \cup E(M_e)) \neq \emptyset$, these edge sets are disjoint for each leaf, and $F$ has at least $\Delta(F)$ leaves. Thus, for any edge $e=uv$, initially $|L \cap (E(G_u) \cup E(M_e) \cup E(G_v))| \leq k+\Delta(F)-(\Delta(F)-1)=k+1$. Furthermore, since $P$ is a matching, if we project a color prescription into a subgraph $G_x$, its origin from a specific $M_e$ is unique. Although $|L \cap (E(G_u) \cup E(M_e) \cup E(G_v))|$ may increase, if we treat the edges $(x,u)(y,u)$ and $(x,v)(y,v)$ as a single unit, the number of constraints does not change effectively (one color prescription is replaced by two equivalent ones). Utilizing the facts that $G$ is triangle-free (implying any precolored edge is incident to at most one neighbor of $x$), $r$-regular, and $k < r$, there must exist a neighbor $y_x$ of $x$ such that no precolored edge is incident to it in either $G_u$ or $G_v$. We then precolor the edges $g_u$ and $g_v$ associated with $x$ with color $c$. This process can be repeated iteratively to project all color prescriptions such that no vertex is incident to more than one precolored edge originating from the projection.

    Since $P$ contains an edge $e=uv$ not incident to a leaf, the two components formed by the removal of $M_e$ and the projection of its color prescriptions contain at most $k+\Delta(F)-1$ precolored edges in their Cartesian product with $G$. This is because each component formed by removing $e$ contains a leaf, and the precolored edges associated with those leaves (which exist by our current case assumption) will not be projected into the other component. Consequently, after projecting the color prescriptions of $M_F$, each $G \square F'_j$ component contains at most $k+\Delta(F)-1$ precolored edge.

    We apply the induction hypothesis and, if necessary, Lemma \ref{extra-color} to color each component $G \square F'_j$. We use the same $\chi'(G)+\Delta(F)-1$ colors for all components. The edges in $M_F$ are colored with the remaining available colors. (There is a color that we did not precolor any of the edges with.) This results in a proper coloring of the graph that satisfies all original precolorings except for those in $M_F$, while also satisfying the projected color prescriptions. The squares defined by the projected color prescriptions, $((x,u),(y_x,u),(x,v),(y_x,v))$, are disjoint because $P$ (and thus $M_F$) is a matching, and each $y_x$ was chosen to be free of other projected color prescriptions. Each such square is colored using exactly two colors. By swapping the colors on the edges along these squares, the coloring remains proper and satisfies the original precolorings on $L \setminus M_F$. Furthermore, this swap ensures that the precolored edges in $M_F$ receive their designated colors, thus providing the required extension of the precoloring $\varphi$.
\end{proof}

\section{Subcubic graphs}

Casselgren, Petros, and Fufa~\cite{cartesian} noted that any partial coloring of a subcubic Class 2 graph with at most three precolored edges is extendable to a proper 4-edge-coloring. For the sake of completeness, we provide a short proof of this statement for all subcubic graph, and then establish a statement which implies Conjecture~\ref{basic} for subcubic Class 2 graphs.

\begin{theorem} \label{3-subcubic}
    Suppose that at most three edges of a subcubic graph $G$ are properly precolored. Then the precoloring can be extended to a proper 4-edge-coloring of $G$.
\end{theorem}

\begin{proof}
    We proceed by induction on the number of edges of $G$. The base case, where $G$ has at most four edges, is trivial, since four colors are available. Assume that the statement holds for all graphs with fewer than $|E(G)|$ edges.

    Suppose that there exists an uncolored edge $e$ incident with a vertex of degree at most two. By the induction hypothesis, $G - e$ admits a proper 4-edge-coloring that extends the precoloring. When $e$ is reinserted, its endpoints are incident with at most three colored edges in total. Thus, at least one of the four colors remains available for $e$. Assigning this color completes the coloring, so we may assume that every uncolored edge is incident only with vertices of degree three. If the number of edges of $G$ is $5$, we are certainly in this case, since at most one edge can have both endpoints of degree three, but there are two uncolored edges. Thus, the statement holds even if $|E(G)| = 5$.

    Suppose $G$ contains a cycle $C$ edge-disjoint from the precolored edges. By removing $C$ and applying the induction hypothesis, the remaining graph admits a proper 4-edge-coloring. Each edge of $C$ then has at least two available colors, as at most one color is forbidden at each endpoint by the edges outside $C$. Since even cycles are 2-edge-choosable, if $C$ is an even cycle, then this coloring extends to $C$.

    Assume that $C$ is an odd cycle. If all edges incident to $C$ but not belonging to it are assigned the same color (say, color $1$), then each edge of $C$ has three available colors excluding $1$. In this case, the edges of $C$ can be colored greedily. Otherwise, these incident edges do not all share the same color. If every edge of $C$ were assigned the same list $\{1, 2\}$, then for any edge $e=uv$ of $C$, the edges incident to $u$ and $v$ outside the cycle would exclude colors $3$ and $4$. Assigning each vertex of $C$ the color of its incident non-cycle edge would then yield a proper 2-coloring of the vertices of $C$, implying that $C$ is even - a contradiction. Hence, the lists on the edges of $C$ are not identical. Since an odd cycle with lists of size at least two that are not all identical is edge-colorable, the coloring extends to $C$.

    Finally, assume that the set of uncolored edges induces a forest $F$. Since each vertex of $F$ in $G$ has degree three, every leaf $v \in F$ is incident to exactly two precolored edges. Given that there are at most three precolored edges and each provides at most two incidences with $V(F)$, the forest $F$ can have at most three leaves.

    As every non-trivial component of a forest contains at least two leaves, $F$ must have exactly one such component, $T$. This component $T$ contains at most one vertex of degree three; otherwise, $T$ would have at least four leaves. Furthermore, if $T$ contains a vertex of degree three, it cannot contain any vertex of degree two in $F$, as such a vertex would require an additional incident precolored edge.

    Consequently, $T$ is either a star $K_{1,3}$ or a path $P_n$ with $n \leq 4$. If $T \cong K_{1,3}$ or $T \cong P_4$, then $G$ is $K_4$. If $T \cong P_4$ where $n < 4$, then $G$ has at most 5 edges; thus, the statement holds in all cases. So, in all cases, the precoloring extends to a proper 4-edge-coloring of $G$.
\end{proof}

\begin{lemma} \label{3-coloring}
    If the vertices of a subcubic graph $G$ can be partitioned into two sets, $U$ and $V$, such that the subgraph induced by $V$, denoted $G[V]$, contains at most one edge and every vertex in $U$ has a degree of at most two in $G$, then $G$ is 3-edge-colorable.
\end{lemma}

\begin{proof}
    Let $B$ be the bipartite subgraph of $G$ containing all edges with one endpoint in $U$ and the other in $V$. If $G[V]$ contains an edge, let it be $e_V = vw$; otherwise, let $e_V = \emptyset$.

    Let $T \subseteq V$ be the set of degree-3 vertices in $V$. Let $B'$ be the subgraph of $B$ induced by $T \cup U$. Every $u \in U$ has a degree at most two in $B$ and so in $B'$ as well. In $G$, every vertex $v \in T$ has degree 3. Since $G[V]$ contains at most one edge, $v$ has at least 2 neighbors in $U$. Therefore, in $B'$, every vertex in $S \subseteq T$ has degree at least 2. The number of edges between $S$ and $N(S)$ in $B'$ is at least $2|S|$. Since every vertex in $N(S) \subseteq U$ has degree at most 2 in $B'$, they can be incident to at most $2|N(S)|$ of these edges. Hence $2|S| \le 2|N(S)|$, which implies $|N(S)| \ge |S|$, so Hall's Condition is satisfied for $T$. Thus, there exists a matching $M_T$ in $B'$ (and so in $B$ as well) that covers every vertex in $T$.

    Construct a new matching $M$ as follows:
    \begin{enumerate}
        \item Start with $M_T$. If $e_V = vw$ exists, remove from $M_T$ the edges incident to $v$ or $w$ (if there are such edges), and add $e_V$ to the set.
        \item Extend this set by greedily adding as many edges from $G[U]$ as possible so that the resulting set $M$ remains a matching.
        \item Color the edges of $M$ with color $1$ and remove them from $G$.
    \end{enumerate}

    In the remaining graph $G' = G \setminus M$, the maximum degree $\Delta(G') \le 2$. Indeed, every $v \in T$ was covered by $M$, so its degree decreased from at most three to at most two. Vertices of $U$ have a degree at most 2 in $G$, and those covered by $M$ have a degree at most 1 in $G'$.

    We show that $G'$ contains no odd cycles. Since $G \setminus \{e_V\}$ is bipartite between $U$ and $V$, except for the edges in $G[U]$, any odd cycle in $G'$ must use at least one edge $e_U = u_1u_2$ from $G[U]$ that was not included in $M$. However, $e_U$ was excluded only if at least one of its endpoints (say $u_1$) was already covered by $M$, which implies that $d_{G'}(u_1) \le 1$. Since a vertex in a cycle must have degree two, $u_1$ (and thus $e_U$) cannot be part of any cycle in $G'$.

    Since $G'$ is a graph with $\Delta(G') \le 2$ and no odd cycles, it is 2-edge-colorable. Combined with the first color used for $M$, we obtain a proper 3-edge-coloring of $G$.
\end{proof}

\begin{lemma} \label{2-edge_4-color}
    Let $G$ be a subcubic graph and $K_2$ have vertices $V(K_2) = \{a, b\}$. Let $e_a = (u, a)(v, a)$ and $f = (w, a)(w, b)$ be two edges in $G \square K_2$ (where $w$ may or may not be in $\{u, v\}$, so $e_a$ and $f$ may or may not be incident). Any proper partial edge-coloring of $\{e, f\}$ can be extended to a proper 4-edge-coloring of the entire graph $G \square K_2$.
\end{lemma}

\begin{proof}
    Let $G_a$ and $G_b$ denote the two copies of $G$ in $G \square K_2$ corresponding to the vertices $a$ and $b$ of $K_2$. Let $e=uv$ be the edge in $G$ corresponding to $e_a$. We denote the color of $e_a$ as $c_e$ and the color of $f$ as $c_f$.

    The proof consists of two cases: $c_e \neq c_f$ and $c_e=c_f$. In both cases, we may assume that the degree of the vertex $w$ and its neighbors in $G$ is exactly $3$. If their degrees were lower, we could attach leaf vertices to $w$ and its neighbors until their degrees reached $3$. The resulting graph $G'$ remains subcubic and, if the precoloring can be extended to a proper 4-edge-coloring of $G' \square K_2$, its restriction to $G \square K_2$ yields a suitable extension for $G \square K_2$. (Removing edges and vertices preserves the propriety of the coloring.)

    \vspace{0.2 cm}

    \noindent \textbf{Case 1: The precolored edges have different colors ($c_e \neq c_f$)}

    \vspace{0.2 cm}

    If $w \in \{u, v\}$, the edges $e_a$ and $f$ are incident. Since they are precolored with different colors, we can trivially satisfy the condition by permuting the colors of a proper 4-edge-coloring (the construction in the next paragraph shows that a 4-edge-coloring exists).

    If $w \notin \{u, v\}$, our goal is to find a 4-edge-coloring $\phi$ of $G$ such that $\phi(e) = c_e$ and some edge incident to $w$ is also colored $c_e$. Given such a coloring, we apply it to both $G_a$ and $G_b$. For any vertex $s \in V(G)$, the edges incident to $(s, a)$ in $G_a$ and to $(s, b)$ in $G_b$ are assigned the same three colors. This leaves one color available for each edge $f_s = (s, a)(s, b)$. Since an edge incident to $w$ in $G$ is colored $c_e$, the color $c_e$ is unavailable for the vertical edge $f = (w, a)(w, b)$. By permuting the three colors other than $c_e$ if necessary, we can ensure that $f$ is assigned $c_f$.

    To achieve this, choose an edge $g \in E(G)$ incident to $w$ such that $g$ is not adjacent to $e$. Since $d(w)=3$, such an edge $g$ is guaranteed to exist. Let $M$ be a maximal matching in $G$ containing both $e$ and $g$. We assign color $c_e$ to all edges in $M$ and remove them from $G$. Let $U = V(M)$ and $V = V(G) \setminus U$. In the remaining graph $G' = G \setminus M$, every vertex in $U$ has degree at most 2, and $V$ is an independent set (by the maximality of $M$). Thus, by Lemma~\ref{3-coloring}, $G'$ is 3-edge-colorable. This yields a proper 4-edge-coloring of $G$ in which $c_e$ is assigned to both $e$ and an edge incident to $w$.

    \vspace{0.2 cm}

    \noindent \textbf{Case 2: The precolored edges have the same color ($c_e = c_f = c$)}

    \vspace{0.2 cm}

    In this case, we need to find a 4-edge-coloring $\phi$ of $G$ such that $\phi(e) = c$, but no edge incident to $w$ is colored $c$. Given such a coloring, we apply it to both $G_a$ and $G_b$. For any vertex $s \in V(G)$, the edges incident to $(s, a)$ in $G_a$ and to $(s, b)$ in $G_b$ are assigned the same three colors. This leaves one color available for each edge $f_s = (s, a)(s, b)$. Furthermore, at vertices $(w, a)$ and $(w, b)$, the color $c$ does not appear on any edge within $G_a$ or $G_b$. Thus, $c$ is the unique available color for the vertical edge $f = (w, a)(w, b)$.

    This requires finding a matching $M$ in $G$ such that $e \in M$, $w \notin V(M)$, and the remaining graph $G \setminus M$ is 3-edge-colorable. According to Lemma~\ref{3-coloring}, $G \setminus M$ is 3-colorable if the subgraph induced by the unmatched vertices $V = V(G) \setminus V(M)$ contains at most one edge. (While not necessary, this condition is sufficient.)

    Since we require $w \in V$, it follows from the maximality of $M$ that any edge in $G[V]$ must be incident to $w$. Since $d_G(w)=3$, $G[V]$ has at most one edge if and only if $M$ covers at least two of the three neighbors of $w$. Let $N(w) = \{x, y, z\}$. Thus, we seek a matching $M$ in $G \setminus \{w\}$ such that $e \in M$ and $|V(M) \cap \{x, y, z\}| \ge 2$.

    \begin{enumerate}
        \item If both endpoints of $e$ are in $N(w)$, then any matching containing $e$ is necessarily suitable.
        \item If exactly one of the endpoints of $e$ is in $N(w)$ (say $x$), then two cases are possible. If $yz$ is an edge, then let $M$ be a maximal matching in $G \setminus \{w\}$ containing both $e$ and $yz$. If $yz \notin E(G)$, then the neighborhood of $\{y,z\}$ contains an element outside $\{ w, u, v \}$. Thus, there exists an edge $h$ such that $\{e,h\}$ is a matching in $G \setminus \{w\}$ that covers two vertices in $N(w)$. In this case, let $M$ be a maximal matching in $G \setminus \{w\}$ containing $\{e, h\}$ as a subset.
        \item If neither endpoint of $e$ is in $N(w)$, we have more options. If $N(w)$ is not an independent set, let $g$ be an edge induced by $N(w)$ and let $M$ be a maximal matching in $G \setminus \{w\}$ containing $\{e, g\}$ as a subset. Thus, we may assume that $N(w)$ is independent. If there are two vertices in $N(w)$ (say, $x$ and $y$) such that their combined neighborhood contains at least two vertices outside the set $\{u, v, w\}$, and both of them have a neighbor outside this set, then there exists a matching $M'$ in $G \setminus \{u, v, w\}$ that covers $\{x, y\}$. Let us extend $M' \cup \{e\}$ to a maximal matching $M$ in $G \setminus \{w\}$. If no such vertices exist, a straightforward case analysis shows that this yields (up to isomorphism) the following three configurations (shown also in Figure~\ref{fig:end_cases}).
        \begin{enumerate}[label=(\alph*)]
            \item If there is a vertex (say $z$) that has two neighbors outside the set $\{u, v, w\}$, then $x$ and $y$ can only have neighbors in $\{u,v\}$.
        \end{enumerate}

        If every vertex has at most one neighbor outside the set $\{u, v, w\}$, then at least two vertices must have such a neighbor (since out of the six edges leaving $N(w)$ in $G \setminus \{w\}$, at most four can go to $\{u, v\}$), and this neighbor is necessarily the same vertex for $x$, $y$ and $z$. Based on this, the other two possible cases are:

        \begin{enumerate}[label=(\alph*)]
        \setcounter{enumii}{1}
            \item One vertex (say $x$) has only neighbors in $\{u,v\}$, while $y$ and $z$ each have one neighbor in $\{u,v\}$ and one common neighbor outside the set $\{u, v, w\}$.

            \item All three vertices $x$, $y$ and $z$ have one neighbor in $\{u,v\}$ and share one common neighbor outside the set $\{u, v, w\}$.
        \end{enumerate}
        
        In these last two cases, we proceed differently than in the previous ones. We use the notation shown in Figure~\ref{fig:end_cases}. The coloring shown in Figure~\ref{fig:end_cases} defines a proper edge-coloring of $G_2$ and the edge between $G_1$ and $G_2$, respecting the color of $e$ such that no edge incident to $w$ is colored with $c$ (the colors used are $\{c,c_1, c_2,c_3\}$). We can color $G_1$ using Vizing's theorem with four colors. By permuting the colors if necessary (within $G_1$ only), we can ensure that the edge between $G_1$ and $G_2$ does not conflict with the two adjacent edges in $G_1$.
    \end{enumerate}

    \begin{figure}
        \centering

        \begin{subfigure}{0.45\textwidth}
            \centering
            \resizebox{\textwidth}{!}{
            \begin{tikzpicture}

            \tikzset{v/.style={circle, fill=black, inner sep=1.5pt}}

            \node[v] (z) at (-3,0.5) {};
            \node[v] (o) at (-4,-1) {};
            \node[v] (b) at (-2,-1) {};

            \node[left] at (z) {$z$};
            \node[left] at (o) {};

            \draw (z) -- (o);
            \draw (z) -- (b);

            \draw (-3,-0.5) circle (1.6);
            \node at (-3,-1.7) {\large $G_1$};

            \node[v] (w) at (0,2) {};
            \node[v] (y) at (0,0.5) {};
            \node[v] (k) at (0,-1) {};
            \node[v] (x) at (2,0.5) {};
            \node[v] (r) at (2,-1) {};

            \node[above] at (w) {$w$};
            \node[left] at (y) {$y$};
            \node[left] at (k) {$u$};
            \node[right] at (x) {$x$};
            \node[right] at (r) {$v$};

            \draw (w) -- (y);
            \node at (-0.3,1.1) {$c_2$};
            \draw (y) -- (k);
            \node at (2.3,-0.3) {$c_1$};
            \draw (w) -- (x);
            \node at (1.3,1.3) {$c_3$};
            \draw (k) -- (r);
            \node at (1,-1.2) {$c$};
            \draw (x) -- (r);
            \node at (-0.3,-0.3) {$c_1$};
            \draw (y) -- (r);
            \node at (0.7,0.3) {$c_3$};
            \draw (k) -- (x);
            \node at (1.3,0.3) {$c_2$};

            \draw (z) -- (w);
            \node at (-1.3,1.1) {$c_1$};

            \draw (1,0.5) ellipse (1.9 and 2.6);
            \node at (1,-1.7) {\large $G_2$};
            \end{tikzpicture}
            }
            \caption{}
        \end{subfigure}
        \begin{subfigure}{0.45\textwidth}
            \centering
            \resizebox{\textwidth}{!}{
            \begin{tikzpicture}

            \tikzset{v/.style={circle, fill=black, inner sep=1.5pt}}

            \node[v] (z) at (-2,0.5) {};
            \node[v] (o) at (-4,-1) {};
            \node[v] (b) at (-2,-1) {};

            \node[left] at (z) {$z$};

            \draw (z) -- (b);
            \node at (-2.25,-0.3) {$c_2$};
            \draw (o) -- (b);
            \node at (-2.6,-1.2) {$c$};

            \draw (-4,-1) circle (1.1);
            \node at (-4,-1.7) {\large $G_1$};

            \node[v] (w) at (0,2) {};
            \node[v] (y) at (0,0.5) {};
            \node[v] (k) at (0,-1) {};
            \node[v] (x) at (2,0.5) {};
            \node[v] (r) at (2,-1) {};

            \node[above] at (w) {$w$};
            \node[left] at (y) {$y$};
            \node[left] at (k) {$u$};
            \node[right] at (x) {$x$};
            \node[right] at (r) {$v$};

            \draw (y) -- (b);
            \node at (-0.7,0.3) {$c_1$};
            \draw (w) -- (y);
            \node at (-0.3,1.1) {$c_2$};
            \draw (z) -- (k);
            \node at (-1.3,0.3) {$c_3$};
            \draw (w) -- (x);
            \node at (1.3,1.3) {$c_3$};
            \draw (k) -- (r);
            \node at (1,-1.2) {$c$};
            \draw (x) -- (r);
            \node at (2.3,-0.3) {$c_1$};
            \draw (y) -- (r);
            \node at (0.7,0.3) {$c_3$};
            \draw (k) -- (x);
            \node at (1.3,0.3) {$c_2$};

            \draw (z) -- (w);
            \node at (-1.1,1.5) {$c_1$};

            \draw (0.3,0) ellipse (2.8 and 2.6);
            \node at (0.3,-1.7) {\large $G_2$};
            \end{tikzpicture}
            }
            \caption{}
        \end{subfigure}

        \begin{subfigure}{0.45\textwidth}
            \centering
            \resizebox{\textwidth}{!}{
            \begin{tikzpicture}

            \tikzset{v/.style={circle, fill=black, inner sep=1.5pt}}

            \node[v] (z) at (-2,0.5) {};
            \node[v] (o) at (4.6,-1) {};
            \node[v] (b) at (-2,-1) {};

            \node[left] at (z) {$z$};

            \draw (z) -- (b);
            \node at (-2.25,-0.3) {$c_2$};
            \draw (o) -- (r);
            \node at (3.2,-1.2) {$c_1$};

            \draw (4.6,-1) circle (1.1);
            \node at (4.6,-1.7) {\large $G_1$};

            \node[v] (w) at (0,2) {};
            \node[v] (y) at (0,0.5) {};
            \node[v] (k) at (0,-1) {};
            \node[v] (x) at (2,0.5) {};
            \node[v] (r) at (2,-1) {};

            \node[above] at (w) {$w$};
            \node[left] at (y) {$y$};
            \node[left] at (k) {$u$};
            \node[right] at (x) {$x$};
            \node[below] at (r) {$v$};

            \draw (y) -- (b);
            \node at (-0.7,0.3) {$c_3$};
            \draw (w) -- (y);
            \node at (-0.3,1.1) {$c_2$};
            \draw (z) -- (k);
            \node at (-1.3,0.3) {$c_3$};
            \draw (w) -- (x);
            \node at (1.3,1.3) {$c_3$};
            \draw (k) -- (r);
            \node at (1,-1.2) {$c$};
            \draw (x) -- (r);
            \node at (2.3,-0.3) {$c_2$};
            \draw (y) -- (k);
            \node at (0.3,0.2) {$c_1$};
            \draw (b) -- (x);
            \node at (1.3,0) {$c_1$};

            \draw (z) -- (w);
            \node at (-1.1,1.5) {$c_1$};

            \draw (0.3,0) ellipse (2.8 and 2.6);
            \node at (0.3,-1.7) {\large $G_2$};
            \end{tikzpicture}
            }
            \caption{}
        \end{subfigure}

        \caption{Coloring of $G_2$}
    \label{fig:end_cases}
    \end{figure}
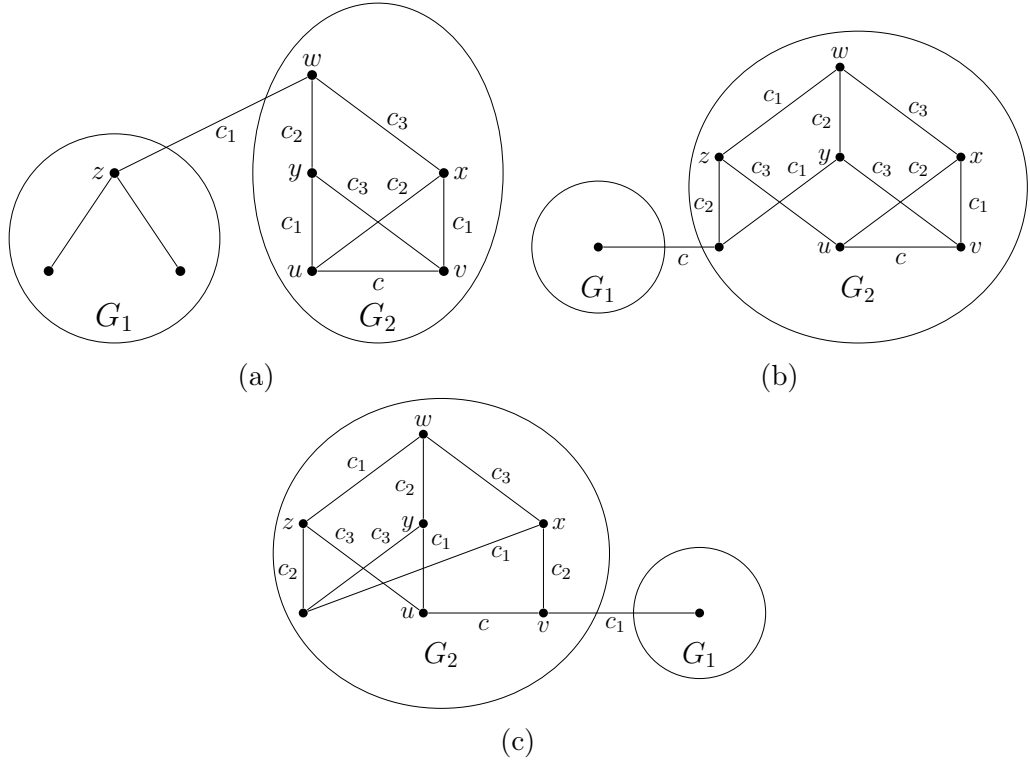
\end{proof}

\begin{theorem} \label{subcubic_Cartesian}
    Suppose that for a subcubic graph $G$, at most four edges of $G \square K_2$ are properly precolored. Then the precoloring can be extended to a proper 5-edge-coloring of $G \square K_2$.
\end{theorem}

\begin{proof}
    Let $V(K_2)=\{a,b\}$. Define $G_a=\{(v,a) \mid v \in V(G)\}$ and $G_b=\{(v,b) \mid v \in V(G)\}$. We call edges of the form $(v,a)(v,b)$ \textit{vertical edges} and edges in $E(G_a) \cup E(G_b)$ \textit{horizontal edges}.

    Suppose first that at least three of the four precolored edges share the same color $c$. In this case, we color the graph with four colors while avoiding $c$. If there is an edge not precolored with $c$ (of which there is at most one), we assign it its prescribed color and ignore the precoloring on the edges colored $c$. This is possible since $\chi'(G \square K_2)=4$ (the Cartesian product of a Class 1 and a Class 2 graph is always Class 1), and the colors in a proper coloring can be permuted arbitrarily.

    Finally, we recolor each edge originally precolored with $c$ back to $c$, obtaining a proper extension of the precoloring. Hence, we may assume that no three precolored edges share the same color.

    Furthermore, we may assume that if $(u,a)(u,b)$ is a precolored edge, then $u$ has degree three in $G$. Indeed, if its degree were smaller, we could add a new vertex $t$ and an edge $ut$; any proper extension of the precoloring on this modified graph would provide a proper extension for the original $G \square K_2$ as well.

    Since we use five colors but only four edges are precolored, at least one color is unused; denote such a color by $c^{\circ}$. In the remainder of the proof, we distinguish cases based on the number of precolored vertical edges.

    \vspace{0.2 cm}

    \noindent \textbf{I. 0 precolored vertical edges}

    \vspace{0.2 cm}

    If there are no vertical precolored edges and both $G_a$ and $G_b$ contain precolored edges, then by Theorem~\ref{3-subcubic} we can color each of $G_a$ and $G_b$ without using the color $c^{\circ}$ while respecting the precolorings. The vertical edges can then all be colored $c^{\circ}$, yielding a proper extension of the precoloring.

    If, on the other hand, all precolored edges lie in $G_a$ (the case for $G_b$ being symmetric), then by Theorem~\ref{3-subcubic} and Lemma~\ref{extra-color}, $G_a$ can be colored with five colors while respecting the precoloring. We then copy this coloring onto $G_b$ (as there are no precolored edges in $G_b$).

    As a result, for each vertical edge, its endpoints are incident to edges of the same colors in $G_a$ and $G_b$, leaving two colors available for the edge itself. Assigning one of these available colors to each vertical edge yields a proper extension of the precoloring.
    
    \vspace{0.2 cm}

    \noindent \textbf{II. 1 precolored vertical edge}

    \vspace{0.2 cm}

    Suppose that exactly one vertical edge is precolored, namely $(s,a)(s,b)$ with color $c$. We consider the following cases.

    \begin{itemize}
        \item If the vertical edge is the only one precolored with $c$, then we color $G_a$ and $G_b$ according to their precolorings without using $c$ (which is possible by Theorem~\ref{3-subcubic}). Assigning color $c$ to all vertical edges then yields a proper extension of the precoloring.
        
        \item If there is another edge, say $(u,a)(v,a)$, precolored with $c$ (of which there is at most one), while $(u,b)(v,b)$ is not precolored, we first temporarily modify the precoloring. Prescribe the color $c^{\circ}$ on both $(u,a)(v,a)$ and $(u,b)(v,b)$. We then color $G_a$ and $G_b$ with four colors, avoiding $c$, while respecting the (new) precolorings (this is possible by Theorem~\ref{3-subcubic}, since each of $G_a$ and $G_b$ has at most three precolored edges). Next, we color all vertical edges with $c$. Finally, consider the 4-cycle $((u,a), (v,a), (v,b), (u,b))$, whose edges are colored with $c$ and $c^{\circ}$. By swapping colors $c$ and $c^{\circ}$ along this cycle, we restore the original precoloring without introducing any conflicts.
        
        \item If $(u,a)(v,a)$ is precolored with $c$ and $(u,b)(v,b)$ with another color $c^*$, but no edge incident to $(u,a)(v,a)$ is precolored with $c^*$, we can color $G_a$ and $G_b$ identically with four colors while avoiding $c$, satisfying all precolorings in $G_a$ and $G_b$ except for $(u,a)(v,a)$. (There are only two precolored edges.) We then recolor $(u,a)(v,a)$ with $c$. At this point, four colors appear at the endpoints of the vertical edges $(u,a)(u,b)$ and $(v,a)(v,b)$, leaving the fifth color available. Coloring these vertical edges with their respective available colors and all other vertical edges with $c$ yields a proper extension of the precoloring.
        
        \item If $(u,a)(v,a)$ is precolored with $c$ and $(u,b)(v,b)$ with another color $c^*$, and there exists an edge $(u,a)(w,a)$ precolored with $c^*$, then there exists an edge $st \in E(G)$ such that $t \notin \{u, v\}$. We temporarily modify the precoloring: assign color $c$ to both $(s,a)(t,a)$ and $(s,b)(t,b)$. Then color $G_a$ and $G_b$ independently with four colors, avoiding $c^{\circ}$, while respecting the precolorings; this is possible by Theorem~\ref{3-subcubic}, since each contains at most three precolored edges. Next, color all vertical edges with $c^{\circ}$. Finally, the cycle $((s,a), (t,a), (t,b), (s,b))$ uses only $c$ and $c^{\circ}$. Swapping these two colors along the cycle preserves proper coloring and restores the original precoloring.
    \end{itemize}

    \vspace{0.2 cm}

    \noindent \textbf{III. 2 precolored vertical edges}

    \vspace{0.2 cm}

    In this part of the proof, let $(u,a)(u,b)$ and $(v,a)(v,b)$ be the two precolored vertical edges with colors $c_u$ and $c_v$, respectively.

    Now, consider the case where $c_u=c_v$; call this common color $c$. We extend the precoloring to $G_a$ and $G_b$ independently using four colors such that color $c$ is avoided in both subgraphs. This is possible by Theorem~\ref{3-subcubic}, since no horizontal edge is precolored with $c$. Assigning $c$ to all vertical edges then yields a proper 5-edge-coloring that extends the original precoloring.

    Suppose now that each of the colors $c_u$ and $c_v$ appears on exactly one precolored edge. There must exist an edge $e \in E(G)$ incident with $u$ or $v$ such that $e \neq uv$ and the horizontal edges corresponding to $e$ are not precolored in either $G_a$ or $G_b$. Without loss of generality, let $e=ux$. We temporarily modify the precoloring: precolor the edges corresponding to $e$ in both $G_a$ and $G_b$ with color $c_u$. This does not create a conflict since $c_u$ was not assigned to any other edge. Next, we extend the precoloring on $G_a$ and $G_b$ independently using four colors while avoiding $c_v$. (this is possible by Theorem~\ref{3-subcubic}, since each graph contains at most three precolored edges). We then assign $c_v$ to all vertical edges. Finally, the cycle $((u,a), (x,a), (x,b), (u,b))$ forms a 4-cycle where the horizontal edges are colored $c_u$ and the vertical edges are colored $c_v$. Swapping these two colors along this 4-cycle preserves the proper coloring and ensures that the vertical edge $(u,a)(u,b)$ receives the required color $c_u$.

    If a color (say $c_v$) is assigned to two precolored edges (necessarily one vertical and one horizontal), we temporarily ignore the precoloring of these edges. We extend the precoloring of the remaining two edges using four colors while avoiding $c_v$, which is possible by Lemma~\ref{2-edge_4-color} (one vertical and one horizontal edge remain). Finally, by restoring $c_v$ to the two edges originally precolored with it, we obtain a proper extension of the entire precoloring.

    \vspace{0.2 cm}

    \noindent \textbf{IV. 3 precolored vertical edges}

    \vspace{0.2 cm}

    In the following two sections (dealing with three and four precolored vertical edges), we apply a new approach. Our goal is to color $G_a$ and $G_b$ identically using five colors such that the horizontal color prescription (if any) is met and the vertical precolorings do not cause conflicts later on.
    
    To achieve this, we define an extended version of $G$, denoted by $G'$ as follows: for each precolored vertical edge $(u,a)(u,b)$, we add a new vertex $u'$ to $G$, join it to $u$, and assign the edge $uu'$ the color prescribed for $(u,a)(u,b)$. In the resulting graph $G'$, three or four vertices have degree four, while all others have degree at most three and there are at most three precolored edges in $G'$.
    
    If this precoloring of $G'$ extends to a proper 5-edge-coloring, then so does the original precoloring of $G \square K_2$. We simply copy this coloring to both $G_a$ and $G_b$, ensuring that the horizontal precolored edge (if any) receives its prescribed color and the precolored vertical edges also keep theirs. Moreover, each vertical edge will have two available colors, as its endpoints share the same three incident colors.

    In each case, our strategy is to select a color $c$ used in the precoloring and find a matching $M$ that covers all degree-4 vertices and includes any edges already assigned that color (there are at most two such edges). We then color all edges in $M$ with color $c$ and remove them from $G'$. Since the resulting graph is subcubic and contains at most three precolored edges, the coloring can be extended using the remaining four colors by Theorem~\ref{3-subcubic}.

    Following the approach described above, we consider the case of three precolored vertical edges. Depending on the number of vertices incident to two precolored edges (which can be 0, 1, or 2), there are three possible scenarios (see Figure~\ref{fig:3-vertical-scenarios}). We examine each case separately and further divide them into subcases based on whether any edges share a color. We use the notation introduced in Figure~\ref{fig:3-vertical-scenarios}.

    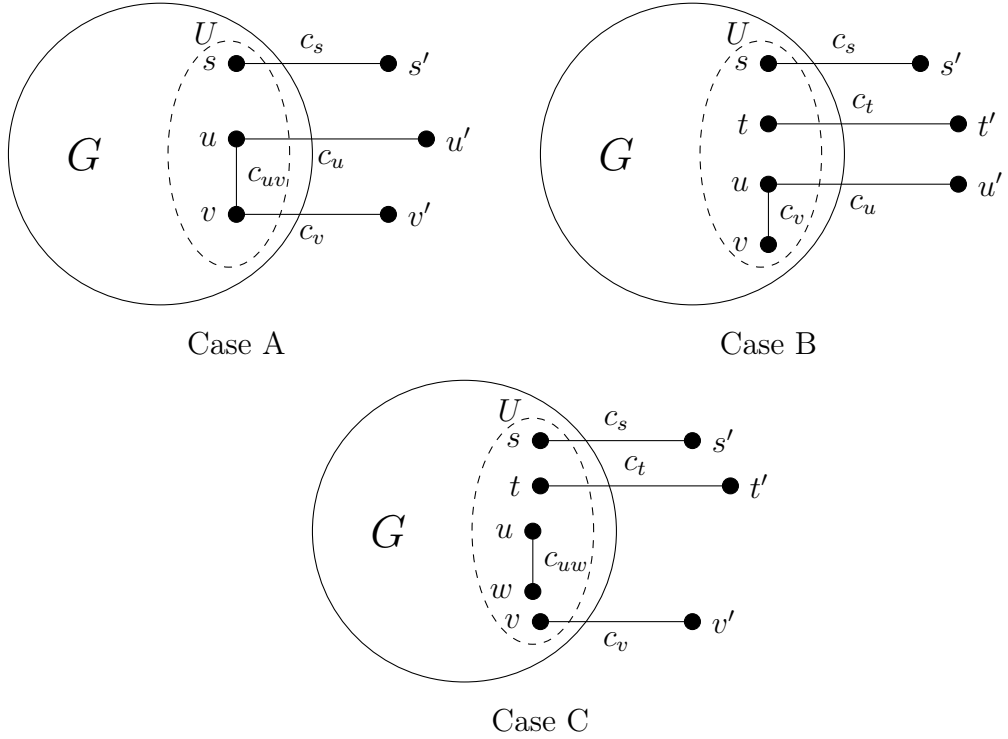
\begin{figure}[h]
        \centering
        \begin{tikzpicture}[
            scale=1,
            vertex/.style={circle, fill=black, draw=black, minimum size=6pt, inner sep=0pt},
        ]

            \draw (0,0) circle (2);

            \coordinate (v1) at (1, 1.2);
            \coordinate (v2) at (1, 0.4);
            \coordinate (v3) at (1, -0.4);
            \coordinate (v4) at (1, -1.2);

            \coordinate (u1) at (3, 1.2);
            \coordinate (u2) at (3.5, 0.4);
            \coordinate (u3) at (3.5, -0.4);

            \draw (v1) -- node[above] {$c_s$} (u1);
            \draw (v2) -- node[above] {$c_t$} (u2);
            \draw (v3) -- node[below] {$c_u$} (u3);
            \draw (v4) -- node[right] {$c_v$} (v3);

            \node[vertex, label=left:$s$] at (v1) {};
            \node[vertex, label=left:$t$] at (v2) {};
            \node[vertex, label=left:$u$] at (v3) {};
            \node[vertex, label=left:$v$] at (v4) {};

            \node[vertex, label=right:$s'$] at (u1) {};
            \node[vertex, label=right:$t'$] at (u2) {};
            \node[vertex, label=right:$u'$] at (u3) {};

            \node at (-1,0) {\Large $G$};

            \draw[dashed]
                (0.9,0) ellipse (0.8 and 1.5);

            \node at (0.6,1.6) {$U$};

            \node at (1,-2.5) {Case B};

            \draw (-7,0) circle (2);

            \coordinate (v11) at (-6, 1.2);
            \coordinate (v13) at (-6, 0.2);
            \coordinate (v14) at (-6, -0.8);

            \coordinate (u11) at (-4, 1.2);
            \coordinate (u13) at (-3.5, 0.2);
            \coordinate (u14) at (-4, -0.8);

            \draw (v11) -- node[above] {$c_s$} (u11);
            \draw (v14) -- node[right] {$c_{uv}$} (v13);
            \draw (v13) -- node[below] {$c_u$} (u13);
            \draw (v14) -- node[below] {$c_v$} (u14);

            \node[vertex, label=left:$s$] at (v11) {};
            \node[vertex, label=left:$u$] at (v13) {};
            \node[vertex, label=left:$v$] at (v14) {};

            \node[vertex, label=right:$s'$] at (u11) {};
            \node[vertex, label=right:$u'$] at (u13) {};
            \node[vertex, label=right:$v'$] at (u14) {};

            \node at (-8,0) {\Large $G$};

            \draw[dashed]
                (-6.1,0) ellipse (0.8 and 1.5);

            \node at (-6.4,1.6) {$U$};

            \node at (-6,-2.5) {Case A};

            \draw (-3,-5) circle (2);

            \coordinate (v1) at (-2, -3.8);
            \coordinate (v2) at (-2, -4.4);
            \coordinate (v3) at (-2.1, -5);
            \coordinate (v5) at (-2.1, -5.8);
            \coordinate (v4) at (-2, -6.2);

            \coordinate (u1) at (0, -3.8);
            \coordinate (u2) at (0.5, -4.4);
            \coordinate (u4) at (0, -6.2);

            \draw (v1) -- node[above] {$c_s$} (u1);
            \draw (v2) -- node[above] {$c_t$} (u2);
            \draw (v3) -- node[right] {$c_{uw}$} (v5);
            \draw (v4) -- node[below] {$c_v$} (u4);

            \node[vertex, label=left:$s$] at (v1) {};
            \node[vertex, label=left:$t$] at (v2) {};
            \node[vertex, label=left:$u$] at (v3) {};
            \node[vertex, label=left:$v$] at (v4) {};
            \node[vertex, label=left:$w$] at (v5) {};

            \node[vertex, label=right:$s'$] at (u1) {};
            \node[vertex, label=right:$t'$] at (u2) {};
            \node[vertex, label=right:$v'$] at (u4) {};

            \node at (-4,-5) {\Large $G$};

            \draw[dashed]
                (-2.1,-5) ellipse (0.8 and 1.5);

            \node at (-2.4,-3.4) {$U$};

            \node at (-2,-7.5) {Case C};

        \end{tikzpicture}
        \caption{Scenarios of three vertical edges}
        \label{fig:3-vertical-scenarios}
    \end{figure}

    \textit{Case A} is straightforward. Since $d(s)=3$, at least one edge incident to $s$, let us call it $e$, has an endpoint in $G \setminus U$. If $c_{uv}=c_s$, then $M = \{uv,ss'\}$ is a suitable matching; otherwise, $M = \{uv, e\}$ is.

    In \textit{Case B}, we distinguish two subcases. If another edge is precolored with $c_u$ (where we may assume $c_s=c_u$), then since $d(t)=3$, there exists an edge $e$ incident to $t$ whose other endpoint is distinct from both $u$ and $s$. In this case, we take the matching $M = \{uu', ss', e\}$. If no other edge is precolored with $c_u$, then both $s$ and $t$ have at least two neighbors in $G \setminus \{u\}$; thus, there exists a matching $M'$ covering $\{s,t\}$ that is disjoint from $u$. (This matching can be the edge $st$, if it exists, or a pair of independent edges.) We then use the matching $M = M' \cup \{ \ {uu'} \ \}$.

    In \textit{Case C}, we consider the following subcases. If two edges other than $uw$ are precolored with the same color (say, $c_s=c_t$), there exists an edge $e$ in $G$ incident to $v$ whose other endpoint is neither $s$ nor $t$. Here, we take the matching $M = \{ss', tt', e\}$.
    
    If $ss', tt'$, and $vv'$ are all precolored distinctly and $\{s, t, v\}$ is not an independent set (say, $st$ is an edge), we take the matching $M = \{st, vv'\}$. If $\{s, t, v\}$ is an independent set, we may assume $c_v \neq c_{uw}$. Since $s$ and $t$ each have at least two neighbors in $G \setminus \{v\}$, there exists a matching $M'$ covering $\{s, t\}$ that is disjoint from $v$. We then take the matching $M = M' \cup \{vv'\}$.

    \vspace{0.2 cm}

    \noindent \textbf{V. 4 precolored vertical edges}

    \vspace{0.2 cm}

    Suppose now that all four precolored edges are vertical. Let $s, t, u, v$ be their endpoints in $G$, and set $U=\{s, t, u, v\}$. Let $c_s, c_t, c_u, c_v$ denote the corresponding colors (not necessarily distinct). The extended graph $G'$ is shown in Figure \ref{fig:graf}. If $G \cong K_4$, the claim follows immediately. (Ellingham and Goddyn \cite{elgo} proved that if $G$ is a $d$-regular $d$-edge colorable planar graph, then $G$ is $d$-edge-choosable, so $K_4$ is 3-edge-choosable.)

    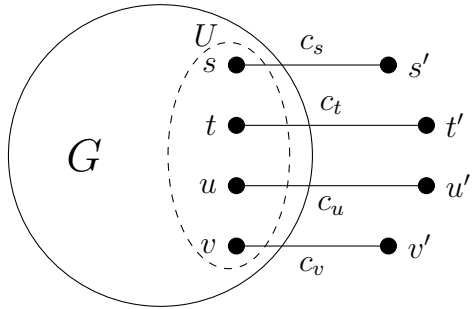
\begin{figure}[h]
        \centering
        \begin{tikzpicture}[
            scale=1,
            vertex/.style={circle, fill=black, draw=black, minimum size=6pt, inner sep=0pt},
        ]

            \draw (0,0) circle (2);

            \coordinate (v1) at (1, 1.2);
            \coordinate (v2) at (1, 0.4);
            \coordinate (v3) at (1, -0.4);
            \coordinate (v4) at (1, -1.2);

            \coordinate (u1) at (3, 1.2);
            \coordinate (u2) at (3.5, 0.4);
            \coordinate (u3) at (3.5, -0.4);
            \coordinate (u4) at (3, -1.2);

            \draw (v1) -- node[above] {$c_s$} (u1);
            \draw (v2) -- node[above] {$c_t$} (u2);
            \draw (v3) -- node[below] {$c_u$} (u3);
            \draw (v4) -- node[below] {$c_v$} (u4);

            \node[vertex, label=left:$s$] at (v1) {};
            \node[vertex, label=left:$t$] at (v2) {};
            \node[vertex, label=left:$u$] at (v3) {};
            \node[vertex, label=left:$v$] at (v4) {};

            \node[vertex, label=right:$s'$] at (u1) {};
            \node[vertex, label=right:$t'$] at (u2) {};
            \node[vertex, label=right:$u'$] at (u3) {};
            \node[vertex, label=right:$v'$] at (u4) {};

            \node at (-1,0) {\Large $G$};

            \draw[dashed]
                (0.9,0) ellipse (0.8 and 1.5);

            \node at (0.6,1.6) {$U$};
        \end{tikzpicture}
        \caption{The extension of the graph $G$}
        \label{fig:graf}
    \end{figure}

    First, consider the case where a color, say $c_s$, appears only once in the precoloring. It follows that another color, say $c_t$, also appears only once. If possible, we choose $s$ or $t$ such that the remaining three vertices in $U$ do not form an independent set (we may assume this vertex is $s$).

    If any of $tu, tv,$ or $uv$ is an edge (say $uv$) and the third vertex $t$ has an incident edge $e$ incident to $G \setminus U$, then the matching $M = \{uv, e, ss'\}$ forms a suitable matching. If $uv \in E(G)$, but $t$ does not have an incident edge to $G \setminus U$, it must be connected to both $u$ and $v$, allowing us to select one of those edges instead of $uv$. If no vertex in $\{t, u, v\}$ has a neighbor in $G \setminus U$, then $G \cong K_4$.

    If $\{t, u, v\}$ is an independent set, consider the bipartite graph between $\{t, u, v\}$ and their neighbors in $G \setminus U$. This graph is subcubic; and each vertex in $\{t, u, v\}$ has degree at least 2, with at least two having degree 3 (as our choice of $s$ ensures $U$ induces at most one edge). By Hall’s condition, there exists a matching $M'$ covering $\{t, u, v\}$. Then $M = M' \cup \{ss'\}$ is a suitable matching.

    The remaining case is when the precolored edges use only two colors, each appearing twice (say, $c_1$ at $u$ and $v$, and $c_2$ at $s$ and $t$). If $uv$ (or $st$) is an edge, or if there exists a matching $M'$ in $G$ covering $\{u,v\}$ (or $\{s,t\}$) such that each edge has an endpoint in $G \setminus U$, then $M = M' \cup \{ss',tt'\}$ (or $M = M' \cup \{uu',vv'\}$) is a suitable matching.
    
    If neither case occurs, then $u$ and $v$ share at most one neighbor, say $x$ in $G \setminus U$, and the same holds for $s$ and $t$ with a neighbor $y$ ($x$ and $y$ are necessarily distinct, since $G$ is subcubic). Moreover, $(s, u, t, v)$ forms a 4-cycle (see Figure~\ref{fig:4-vertical-worst_case}).

    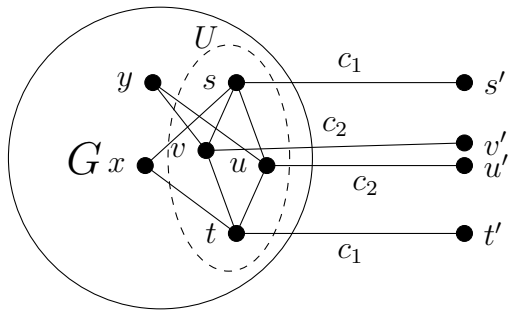
\begin{figure}[h]
        \centering
        \begin{tikzpicture}[
            scale=1,
            vertex/.style={circle, fill=black, draw=black, minimum size=6pt, inner sep=0pt},
        ]

            \draw (0,0) circle (2);

            \coordinate (v1) at (1, 1);
            \coordinate (v2) at (0.6, 0.1);
            \coordinate (v3) at (1.4, -0.1);
            \coordinate (v4) at (1, -1);

            \coordinate (w1) at (4, 1);
            \coordinate (w2) at (4, 0.2);
            \coordinate (w3) at (4, -0.1);
            \coordinate (w4) at (4, -1);

            \coordinate (u1) at (-0.2, -0.1);
            \coordinate (u2) at (-0.1, 1);

            \draw (v1) -- node[above] {} (u1);
            \draw (v2) -- node[above] {} (u2);
            \draw (v3) -- node[below] {} (u2);
            \draw (v4) -- node[below] {} (u1);

            \draw (v1) -- node[above] {} (v2);
            \draw (v2) -- node[above] {} (v4);
            \draw (v3) -- node[below] {} (v4);
            \draw (v3) -- node[below] {} (v1);

            \draw (v1) -- node[above] {$c_1$} (w1);
            \draw (v2) -- node[above] {$c_2$} (w2);
            \draw (v3) -- node[below] {$c_2$} (w3);
            \draw (v4) -- node[below] {$c_1$} (w4);

            \node[vertex, label=left:$s$] at (v1) {};
            \node[vertex, label=left:$v$] at (v2) {};
            \node[vertex, label=left:$u$] at (v3) {};
            \node[vertex, label=left:$t$] at (v4) {};

            \node[vertex, label=right:$s'$] at (w1) {};
            \node[vertex, label=right:$v'$] at (w2) {};
            \node[vertex, label=right:$u'$] at (w3) {};
            \node[vertex, label=right:$t'$] at (w4) {};

            \node[vertex, label=left:$x$] at (u1) {};
            \node[vertex, label=left:$y$] at (u2) {};

            \node at (-1,0) {\Large $G$};

            \draw[dashed]
                (0.9,0) ellipse (0.8 and 1.5);

            \node at (0.6,1.6) {$U$};
        \end{tikzpicture}
        \caption{Neighborhood of $U$ in the case of identical pair of colors on the precolored edges}
        \label{fig:4-vertical-worst_case}
    \end{figure}

    In this case, we define a temporary precoloring on $G$ as follows: let $yv, xs$ be precolored with a color $c_3$, and $yu, xt$ with a color $c_4$ (where $c_3, c_4 \notin \{c_1, c_2\}$). Since $G$ is subcubic and we have four precolored edges, by Theorem~\ref{3-subcubic} and Lemma~\ref{extra-color}, this can be extended to a proper 5-edge-coloring $\phi$ of $G$.
    
    To obtain the final coloring, we modify $\phi$ on the edges of the 4-cycle $(s,v,t,u)$. We set $\phi(tu)=c_3$, $\phi(vs)=c_4$, and $\phi(vt)=\phi(su)=c_5$. This preserves the propriety of the coloring, and the colors $c_1, c_2$ are no longer used on any edges incident to $U$ in $G$. Thus, the edges $ss', tt', uu',$ and $vv'$ can be colored with their prescribed colors $c_1$ and $c_2$ without conflict. This yields a proper 5-edge-coloring of $G'$, which in turn provides a proper extension of the precoloring in $G \square K_2$.
\end{proof}

A natural question arises as to whether Theorem~\ref{subcubic_Cartesian} can be generalized. For example, is it true that if the maximum degree of $G$ is four, then any precoloring of five edges in $G \square K_2$ can be extended to a proper 6-edge-coloring? While the answer is unknown, there is a fact suggesting that this generalization might not hold. Specifically, Theorem~\ref{3-subcubic} cannot be generalized in this manner. Consider, for example, the 4-regular graph $K_5$. If we precolor three edges of a triangle and one additional edge vertex-disjoint from the triangle (four edges in total) using four distinct colors, the precoloring cannot be extended using five colors. Thus, even if Theorem~\ref{subcubic_Cartesian} is generalizable, entirely new approaches would certainly be required.

\section{Concluding remarks}

In this paper, we propose Conjecture \ref{general} as a far-reaching generalization of Conjecture \ref{basic}. Establishing this general conjecture would resolve numerous open problems concerning the extendability of edge-precolorings in Cartesian products. Theorem \ref{edge} may be viewed as an initial step toward this goal.

Several directions for further research naturally arise. One major challenge is to relax the structural assumptions in $G$ in Theorem \ref{edge}. Specifically, it would be of great interest to eliminate the requirements of triangle-freeness and/or regularity. Since both assumptions play a crucial role in the counting arguments and the projection-exchange techniques used in our current proofs, addressing more general classes of graphs will likely require substantially new ideas.

Another promising avenue for future study involves identifying additional classes of graphs $H$ for which Conjecture \ref{general} holds, beyond stars, even cycles, and trees. Although many potential candidates may be considered, complete bipartite graphs stand out as especially intriguing.

\begin{conjecture}
    Let $G$ be an $r$-regular, triangle-free graph. If any precoloring of at most $k < r$ edges of $G$ can be extended to a proper $\chi'(G)$-edge-coloring of $G$, then any precoloring of at most $k+n$ edges of $G \square K_{n,n}$ can be extended to a proper $(\chi'(G)+n)$-edge-coloring of $G \square K_{n,n}$.
\end{conjecture}
    
Establishing this conjecture (or more general for $H = K_{n,m}$ instead of $H = K_{n,n}$) would not only deepen our understanding of the precoloring extension in Cartesian products, but would also imply a solution to the following conjecture of Casselgren, Markström and Pham.

\begin{conjecture} \cite{hypercube}
    If $n$ and $d$ are positive integers, and $\varphi$ is a proper edge-precoloring of $K_{n,n}^d$, with at most $nd - 1$ precolored edges, then $\varphi$ extends to a proper $nd$-edge-coloring of $K_{n,n}^d$.
\end{conjecture}

\section*{Statements and declarations}
The author declares that they have no conflict of interest.
Data sharing is not applicable to this article, as no datasets were generated or analyzed during the current study.

\end{document}